\documentclass{ws-jnmp}

\usepackage[T1]{fontenc}
\usepackage{graphicx}

\usepackage{color}
\definecolor{MyLinkColor}{rgb}{0,0,0.4}

\newcommand{\R}{{\mathbb R}}

\newcommand{\Z}{{\mathbb Z}}
\newcommand{\N}{{\mathbb N}}
\newcommand{\s}{\mathbb S}
\newcommand{\h}{\rho}

\newcommand{\U}{\mathcal{U}}

\newcommand{\A}{\mathcal{A}}
\newcommand{\bA}{\mathbb{A}}
\newcommand{\B}{\mathcal{B}}

\newcommand{\ov}{\overline}
\newcommand{\p}{\partial}
\newcommand{\0}{\Omega}

\newcommand{\tr}{\mathop{\rm tr}\nolimits}
\newcommand{\T}{\mathcal{T}}
\newcommand{\e}{\varepsilon}
\newcommand{\wt}{\widetilde}

\newcommand{\ap}{{0+}}
\newcommand{\im}{\mathop{\rm Im}\nolimits}

\newcommand{\Kern}{\mathop{\rm Ker}\nolimits}
\newcommand{\spn}{\mathop{\rm span}\nolimits}

\begin{document}

\newtheorem{thm}{Theorem}[section]
\newtheorem{prop}[thm]{Proposition}
\newtheorem{cor}[thm]{Corollary}

\arttype{Article} 

\markboth{A.--V. Matioc}
{Internal water waves with a critical layer}

%
%

\title{{\sc Steady internal water waves with a critical layer bounded by the wave surface}}

\author{\footnotesize ANCA--VOICHITA MATIOC}

\address{Institute for Applied Mathematics, Leibniz University of Hanover, Welfengarten 1\\ 30655 Hanover, Germany\\
\email{matioca@ifam.uni-hannover.de}}

\maketitle


\begin{abstract}
In  this paper we construct small amplitude periodic internal waves traveling at the boundary region between two rotational and  homogeneous fluids with different densities. 
Within a period, the waves we obtain have the property that the gradient of the stream function associated to the fluid beneath the interface vanishes, on the wave surface, at exactly two points. 
Furthermore, there exists a critical layer  which is bounded from above by the wave profile.
Besides, we prove, without excluding the presence of  stagnation points, that 
if the  vorticity function associated to each fluid in part is real-analytic, bounded, and non-increasing, then  capillary-gravity steady internal waves  
are a priori real-analytic.
Our new method  provides  the real-analyticity of capillary and capillary-gravity waves with stagnation points traveling over a homogeneous rotational fluid under the same restrictions on the vorticity function.
\end{abstract}


\ccode{2010 Mathematics Subject Classification: 35Q35;  76B45; 76B55, 37N10.}

\section{Introduction}

In this paper we consider two-dimensional internal periodic waves traveling at the  interface between two layers of immiscible fluids with different densities,
 under the rigid lid assumption.
In our context, the fluids have constant vorticity and we construct internal traveling waves with a critical layer and stagnation points in both gravity and capillary-gravity regimes.

A critical layer is a region of  fluid consisting entirely of closed streamlines, and stagnation points are fluid particles traveling horizontally with the same speed as the wave.
Flows with stagnation points are known  \cite{CJ} to be relevant for the description of background states for tsunamis.
Concerning waves traveling over a homogeneous fluid, it was observed in \cite{C2, CS3} that the strong elliptic maximum principle rules out the existence of smooth  irrotational waves with stagnation points or critical layers.
On the other hand, it was shown in \cite{M,P, T} that there exist extreme Stokes  waves which are Lipschitz continuous and have stagnation points  and a sharp corner at the crest.
For the rotational case the picture is different.  
Existence of exact periodic traveling gravity waves with general vorticity has been established first in \cite{CS2} by means of bifurcation and degree
theory, and in \cite{CS41}    in the case of waves with bounded and discontinuous  vorticity.
By construction, the waves in \cite{CS2, CS41} do not possess stagnation points and critical layers. 
Linear gravity water waves with stagnation points which travel on currents with constant vorticity   were studied in   \cite{Mats}, and the exact picture of the waves has been obtained first in \cite{W},  
and later on in \cite{CV2} by employing complex methods.
These waves possess at most a critical layer which is located within the fluid body.

One of the  factors which can determine the presence of critical layers is  the vertical stratification of the fluid. 
It was recently shown in \cite{EMM2} that continuously stratified gravity waves with  constant vorticity may possess two critical layers  and the qualitative picture of the streamlines
may be different than that for homogeneous flows \cite{CV2, W}.
Another factor is the presence of an affine vorticity distribution, situation analysed in \cite{EEV, EEW}, when the waves  may possess arbitrarily many critical layers.
For a survey on traveling waves with critical layers  we refer to the  article \cite{EW}.

Internal waves are usually created by the presence of two different layers of water combined with a certain configuration of relief and current.
They form where the water above and below the interface is either moving in opposite directions or in the same direction at different speeds. 
A well-known example are the internal waves which move from the Atlantic Ocean to the Mediterranean Sea, at the east of the Strait of Gibraltar.
In this case the two layers correspond to different salinities, whereas the current is caused by the tide passing through the strait.
The mathematical theory of waves at the interface between two layers of immiscible fluid of different densities has
attracted a lot of interest and we refer to the  survey article of
Helfrich and Melville \cite{HM} which provides  a good overview on  steady
internal solitary waves in such systems.

In this paper we construct periodic steady internal waves between two layers of homogeneous fluids with different densities in the gravity and capillary-gravity regimes.
  Particularly, we find  solutions which possess exactly one critical layer which is bounded from above by the wave profile.
 Moreover, the gradient of the stream function associated to the fluid beneath the interface vanishes at exactly two points on the interface, meaning that, close to these points, fluid particles located below the interface 
 move almost horizontally with velocity approaching  the wave speed.
The fluid above the interface does not sense the presence of these stagnation points. 
 In contrast to the situation in \cite{EMM2}, where  waves with critical layers were constructed in an unstable regime, the solutions we find are stably  stratified. 
 While in \cite{M,T}   the wave surface is only Lipschitz continuous in a neighbourhood of  the stagnation point,  our solutions are real-analytic.
In the pure gravity case this follows by using a regularity result for free boundary value problems, cf. \cite{KNS}, whereas when capillary plays a role we provide a new idea based on parabolic theory.
It is worthwhile to mention that our method may be applied  to prove  the a priori analyticity of  water waves traveling over an homogeneous fluid, when capillary effects are incorporated, provided the vorticity function is bounded, analytic,  and 
non-increasing even if stagnation points exist (situation complementary to that analysed in  \cite{DH2, Hen2}).

We determine, by using higher order expansions and elliptic  maximum principles, a precise picture of the streamlines in the fluids which could be used to  describe  the particle trajectories.
Due to the analyticity of the interface,  one could choose also linear theory, similarly as in \cite{CV, Mats, IK,  AM}, to obtain an  approximative picture of  the exact particle paths.  

The outline of the paper is as follows: in Section \ref{S:1} we present the mathematical model, re-express the problem in terms of the stream functions and 
provide the main regularity result Theorem \ref{MT:0}.  
In Section \ref{S:2} we write the problem as a  nonlocal equation for the wave profile, and use bifurcation theory to prove the existence statement Theorem \ref{MT:1}
 as well as higher order expansion for the bifurcation curves.
 In the last section, based on Theorem \ref{MT:0}, we illustrate the precise picture of the streamlines for some of the waves obtained in Theorem \ref{MT:0}.

\section{The governing equations and a regularity result}\label{S:1}

In this section we present the governing equations for two-dimensional internal water waves travelling at the 
boundary region between two rotational fluids with different densities under the rigid lid  condition at the top.
The bottom of the ocean is assumed to be flat and we denote with $\s$ the unit circle, i.e. $\s$ stands for $\R/ 2\pi \Z$.
The fluid domain   $\0:=\s\times(-1,1)$ contains two  fluids layers separated by a sharp interface $y=\eta(t,x)$, the wave profile,  which defines the two subsets 
\begin{align*}
 \0_\eta^b:=\{(x,y)\,:\,\text{$ x\in\s $ and $-1<y<{\eta(t,x)}$}\}, \qquad \0^t_\eta:=\{(x,y)\,:\,\text{$ x\in\s $ and ${\eta(t,x)}<y<1$}\}.
\end{align*}
The domain $\0_\eta^b$ contains a  Newtonian fluid with constant
 density $\h,$ velocity field $(u, v)$, and pressure $P$, and we denote by  $\ov \h,$  $(\ov u,\ov  v)$, and  $\ov P$
the density, velocity, and pressure of the fluid located at the top. 
The line $y=-1$ is the impermeable bottom of the ocean and $y=1$ is the rigid top of the two-fluid system.
Assuming that both fluids are inviscid, the dynamics  can be described  by Euler's equations (see \cite{Ligh} for a justification of the inviscid flow).
Being  interested in traveling internal  waves, we   presuppose that there exists a positive constant $c$, the wave speed, such that $\eta(t,x)=\eta(x-ct)$,
\[
 (u,v,P)(t,x,y)=(u,v,P)(x-ct,y),\qquad (\ov u,\ov v,\ov P)(t,x,y)=(\ov u,\ov v,\ov P)(x-ct,y),
\]
and formulate the problem in a frame moving with the wave.
The equations of motion  are the steady state two-dimensional Euler equations
\begin{subequations}\label{eq:2}
\begin{equation}\label{eq:20}
\left\{
\begin{array}{rllll}
(u-c) u_x+vu_y&=&-P_x/\h&\text{in}\ \0_\eta^b,\\
(u-c) v_x+vv_y&=&-P_y/\h-g&\text{in}\ \0_\eta^b,\\
u_x+v_y&=&0&\text{in}\ \0_\eta^b,
\end{array}
\right.
\qquad
\left\{
\begin{array}{rllll}
(\ov u-c) \ov u_x+\ov  v\hspace{0.7mm}\ov u_y&=&-\ov P_x/\ov \h&\text{in}\  \0_{\eta}^t,\\
(\ov u-c) \ov v_x+\ov v\hspace{0.7mm}\ov v_y&=&-\ov P_y/\ov \h-g&\text{in}\ \0_{\eta}^t,\\
\ov u_x+\ov v_y&=&0&\text{in}\  \0_{\eta}^t,
\end{array}
\right.
\end{equation}
subjected to the boundary conditions, see e.g. \cite{BLS, CS2},
\begin{equation}\label{eq:21}
\left\{
\begin{array}{rllll}
 v&=&0&\text{on}\  y=-1,\\
 \ov v&=&0&\text{on}\ y=1,
\end{array}
\right.
\qquad\text{and}\qquad
\left\{
\begin{array}{rllll}
P&=&\ov P-\sigma\eta''/(1+\eta'^2)^{3/2}&\text{on} \ y=\eta(x),\\
v&=&(u-c)\eta'&\text{on}\ y=\eta(x),\\
\ov v&=&(\ov u-c)\eta'&\text{on}\ y=\eta(x),
\end{array}
\right.
\end{equation}
\end{subequations}
with $\sigma\geq0$ being the surface tension coefficient and $g$ the constant of gravity.
To be more precise,  we are interested in finding solutions   of problem \eqref{eq:2} within the class
 \begin{align}\label{regularity}
\eta\in C^{2+\alpha}(\s),\quad  (u,v,P)\in  \left(C^{1+\alpha}\big(\ov{\0_\eta^b}\big)\right)^3,
\quad (\ov u,\ov v,\ov P)\in  \left(C^{1+\alpha}\big(\ov{\0_\eta^t}\big)\right)^3
 \end{align}
for some $\alpha\in (0,1).$
Similarly as in \cite{CS2}, we  reformulate the problem \eqref{eq:2}  by introducing the stream functions 
  $ \psi : {\0_{\eta}^b}\to\R$ and $\ov \psi: \0_{\eta}^t\to\R$ defined by the relations 
\[ 
  \psi(x,y):= m+\int_{-1}^y( u(x,s)-c)\, ds\qquad\text{and}\qquad\ov \psi(x,y):=\int_{\eta(x)}^y( \ov u(x,s)-c)\, ds.
\]
Here $m$  is a positive constant fixed such that $\psi=0$ on $y=\eta(x).$ 
Indeed, taking into account that 
$\nabla \psi=(-v,u-c)$ and $ \nabla\ov\psi=(-\ov v, \ov u-c) $
it follows, by using the chain rule and \eqref{eq:21}, that $\psi$ is constant on $y=\eta(x) $ and we can make this choice.
Similarly,   $\ov\psi$ is constant on the rigid lid $y=1,$ and we let $\ov m:=\psi(0,1)$ be this constant.
The above properties show that the streamlines coincide with the level curves of the stream functions.
Both fluids being  rotational, we introduce their   vorticity    by
\[
\omega:=u_y-v_x=\Delta\psi\quad \text{in $\0_{\eta}^b$}\qquad\text{and}\qquad \ov\omega:=\ov u_y-\ov v_x=\Delta\ov \psi\quad \text{in $\0_{\eta}^t$.}
\]
Next, we assume that 
\begin{equation}\label{eq:cond1}
u-c<0 \ \ \text{in $\ov{\0_\eta^b}$} \qquad\text{and}\qquad \ov u-c<0 \ \ \text{in $\ov{\0_\eta^t}$},
\end{equation}
and exclude so the presence of  stagnation points in the fluids.
However, this assumption \eqref{eq:cond1}  will be dropped later on, after expressing the problem \eqref{eq:2} in terms of the stream functions. 
This will allow us to find solutions of  system \eqref{eq:2} which do possess stagnation points.

Condition \eqref{eq:cond1} guarantees  \cite{CS2, M1}
that the vorticity is a single-valued function of the stream function, that is,
there exist $\gamma\in C^\alpha([-m,0])$ and $\ov\gamma\in C^\alpha([0,-\ov m])$, the   vorticity functions, with
\[
\omega(x,y)=\gamma(-\psi(x,y)),\quad (x,y)\in\0_{\eta}^t\qquad\text{and}\qquad \ov\omega(x,y)=\ov\gamma(-\ov\psi(x,y)),\quad (x,y)\in\0_{\eta}^b.
\]
Since, by Bernoulli's principle, the quantities 
\[
\frac{( u-c)^2+ v^2}{2}+\frac{ P}{ \h}+gy-\int_0^{ \psi} \gamma(-s)\, ds\qquad\text{and}\qquad
\frac{( \ov u-c)^2+ \ov v^2}{2}+\frac{ \ov P}{ \ov \h}+gy-\int_0^{\ov \psi} \ov \gamma(-s)\, ds
\]
are constant in $\0_{\eta}^t$ and $\0_{\eta}^b,$ respectively, we find, by restricting to   $y=\eta(x)$ that  
\begin{equation*}
\frac{|\nabla  \psi|^2}{2}+gy+\frac{ P}{ \h}=const.\qquad\text{and}\qquad \frac{|\nabla \ov \psi|^2}{2}+gy+\frac{\ov P}{\ov \h}=const.
\end{equation*}

To conclude,  we observe that $\psi$ and $\ov\psi$ are solutions of the  semilinear Dirichlet problems
\begin{subequations}\label{eq:PP}
\begin{equation}\label{eq:psiDP}
\left\{
\begin{array}{rllll}
\Delta \psi&=&\gamma(-\psi)&\text{in}\ \0_\eta^b,\\
\psi&=&0&\text{on}\ y=\eta(x),\\
\psi&=&m&\text{on} \ y=-1,
\end{array}
\right.
\qquad\text{and}\qquad
\left\{
\begin{array}{rllll}
\Delta \ov\psi&=&\ov\gamma(-\ov \psi)&\text{in}\ \0_{ \eta}^t,\\
\ov\psi&=&0&\text{on}\ y=\eta(x),\\
\ov\psi&=&\ov m&\text{on} \ y=1,
\end{array}
\right.
\end{equation}
respectively,  and they are coupled by the following equation
\begin{equation}\label{eq:CR}
\ov\h |\nabla \ov \psi|^2-\h|\nabla  \psi|^2+2 g (\ov \h- \h )y+2\sigma\eta''/(1+\eta'^2)^{3/2}=Q\qquad\text{on} \ y=\eta(x)
\end{equation}
for some constant $Q\in\R.$
If we integrate this last equation \eqref{eq:CR} over the unit circle, we obtain, by taking into account that $\eta''/(1+\eta'^2)^{3/2}=(\eta'/(1+\eta'^2)^{1/2})'$, that
\begin{equation}\label{eq:CR1}
Q=\int_\s(\ov\h |\nabla \ov \psi|^2(x,\eta(x))-\h|\nabla  \psi|^2(x,\eta(x))+2 g (\ov \h- \h )\eta(x))\, dx.
\end{equation}
\end{subequations}
The integral over $\s$ is normalised such that $\int_\s1\, dx=1.$ 
It is not difficult to check that given two vorticity functions $\gamma$ and $\ov\gamma$ and a classical solution $(\eta,\psi,\ov\psi)$ of \eqref{eq:PP}, then we can associate to this solution  a unique solution of problem
\eqref{eq:2}, even if the condition \eqref{eq:cond1} is not satisfied.
This formulation of problem \eqref{eq:2} is much more convenient  because we deal with two coupled  semilinear elliptic problems.If $\gamma$ and $\ov \gamma$ are such that we may solve \eqref{eq:psiDP} for given $\eta$,
 which is definitely the case when the fluids have constant vorticity, then \eqref{eq:PP} reduces to the problem of determining $\eta.$
As in \cite{W}, we remark that the solutions of \eqref{eq:PP} are solutions of \eqref{eq:2} for any value $c\in(0,\infty) $ of the wave speed.

As a first result, we use in the pure gravity case $\sigma=0$ a theorem on the regularity  of free interfaces from \cite{KNS},
 which was employed  in \cite{CE3, DH2, DH4, M1} to establish analyticity  of traveling water waves without stagnation points on the profile in different wave  regimes,
 to prove that internal traveling waves
with real-analytic vorticity  functions are a priori real-analytic.
 When we incorporate surface tension effects, we provide a new idea, which is based on parabolic theory, and show the real-analyticity without excluding the  existence of stagnation points.
 
\begin{thm}\label{MT:0} Let  $\sigma\in[0,\infty)$ and assume that $\gamma$ and $\ov\gamma$ are constant functions. 
Given a solution  $(\eta,\psi,\ov\psi)$  of \eqref{eq:PP} within the class \eqref{regularity}, we assume,
when $\sigma=0$,  that 
\begin{equation}\label{eq:CD}
 |\nabla \psi|^2+|\nabla\ov\psi|^2>0 \qquad\text{on $y=\eta(x)$}.
\end{equation}
Then, the wave profile $\eta$ is  real-analytic $\eta\in C^\omega(\s).$
\end{thm}
\begin{remark} It is clear from the proof of Theorem \ref{MT:0} that its assertion is true in the case  $\sigma=0$ for arbitrary analytic vorticity functions $\gamma, \ov\gamma,$ provided \eqref{eq:CD}
  is fulfilled.
  When $\sigma>0,$ and $\gamma,\ov\gamma\in C^\omega(\R)$ are bounded and satisfy additionally
  \begin{equation}\label{eq:adas}
\gamma'\leq0\qquad\text{and}\qquad\ov \gamma'\leq0,
\end{equation}
then the claim of Theorem \ref{MT:0} is still true. 
Relation \eqref{eq:adas} ensures the unique solvability of the Dirichlet problems \eqref{eq:psiDP}, see Theorem 3.3 in \cite{DB}, which permits us to express the problem \eqref{eq:PP} by the equation \eqref{OE}.
Our method may be used  to prove analyticity of the profile for traveling water waves  with  stagnation points in the capillary and capillary-gravity  regime, 
 provided the vorticity function is non-increasing, cf. \cite{DH4, Hen2}.
\end{remark}
\begin{proof} Assuming  $\sigma=0$ first, we  re-write the coupling equation \eqref{eq:CR} in a different way.
Since the boundary conditions on $y=\eta(x)$ of both problems \eqref{eq:psiDP} imply that the
 tangential component of  $\nabla \psi$ and $\nabla \ov\psi$ vanish at the wave profile, this implies
\[
|\nabla \psi|^2=(\p_\nu\psi)^2\qquad\text{and}\qquad |\nabla \ov\psi|^2=(\p_\nu\ov\psi)^2\qquad\text{on $y=\eta(x)$}.
\]
Here, $\nu:=(-\eta',1)/(1+\eta'^2)^{1/2}$  is the unit normal at $y=\eta(x).$
With this notation, equation \eqref{eq:CR} is equivalent to the following relation
\[
f(x,y,\p_\nu\ov\psi,\p_\nu\psi)=\ov \h(\p_\nu\ov\psi)^2-\h(\p_\nu\psi)^2+2g(\ov \h-\h)y-Q=0\qquad\text{on $y=\eta(x)$}.
\]
Clearly, $f$ is real-analytic in all its arguments.
Furthermore, our assumption \eqref{eq:CD} implies that
\begin{align*}
f_{\p_\nu\ov\psi}\p_\nu\ov\psi-f_{\p_\nu\psi}\p_\nu\psi=2(\ov \h(\p_\nu\ov\psi)^2+\h(\p_\nu\psi)^2)\neq0
\end{align*}
on $y=\eta(x),$ which shows that the assumptions of Theorem 3.2  in \cite{KNS} (see also the Remark following it) are all satisfied and the claim follows at once.

The proof in the case when $\sigma\neq0$ is different. 
Indeed, we can  reduce the problem \eqref{eq:PP} (see e.g. Section \ref{S:2}) to an operator equation
\begin{equation}\label{OE}
\frac{\eta''}{(1+\eta'^2)^{3/2}}+\phi(\eta)=0,
\end{equation}
where setting 
\begin{equation}
\label{U}\U:=\{\eta\in C^{2+\alpha}(\s)\,:\, |\eta|<1\},
\end{equation} 
we have that $\phi:\U\to C^{1+\alpha}(\s)$ is a real-analytic operator with $\phi(C^\infty(\s))\subset C^\infty(\s).$
Because we use parabolic theory later on in the proof, we introduce the small H\"older spaces $h^{k+\alpha}(\s), $ $k\in\N,$  as the closure of the smooth functions $C^\infty(\s)$  in $C^{k+\alpha}(\s).$
We define now $D:=\U\cap h^{2+\alpha}(\s)$, and observe that \eqref{OE} may be written as
\begin{equation}\label{OE1}
\eta'=\frac{\eta''}{(1+\eta'^2)^{3/2}}+\varphi(\eta),
\end{equation}
with $\varphi\in C^\omega(D, h^{1+\alpha}(\s))$ given by $\varphi(\eta):=\eta'+\phi(\eta),$ $\eta\in D.$
Indeed,  one can easily see from \eqref{OE1} that $\eta\in C^\infty(\s),$ which implies $\eta\in D.$

Let now $\xi:\R\to h^{2+\alpha}(\s)$ be the function defined by $\xi(t,x)=\eta(t+x).$
If we remark that $\varphi(\eta(t+\cdot))=\varphi(\eta)(t+\cdot)$ for all $t\in\R $, which is a direct consequence of the unique solvability of \eqref{eq:psiDP} and of the fact 
that the variable $x$ does not interfere into \eqref{eq:psiDP},    we find
\begin{align*}
\p_t\xi(t,x)=&\p_x\eta(t+x)=\frac{\p_x^2\xi}{(1+(\p_x\xi)^2)^{1/2}}(t,x)+\varphi(\eta)(t+x)\\
=&\frac{\p_x^2\xi}{(1+(\p_x\xi)^2)^{1/2}}(t,x)+\varphi(\eta(t+\cdot))(x)=\frac{\p_x^2\xi}{(1+(\p_x\xi)^2)^{1/2}}(t,x)+{\varphi}(\xi(t))(x).
\end{align*} 
Whence,  $\xi$ is a solution of the autonomous  problem 
\begin{equation}\label{CP}
\p_t\xi=\Phi(\xi), \quad t>0,\qquad \xi(0)=\eta, 
\end{equation}
where $\Phi\in C^\omega(D, h^{\alpha}(\s))$ is the mapping
\[
\Phi(\xi)=\frac{\p_x^2\xi}{(1+(\p_x\xi)^2)^{1/2}}+{\varphi}(\xi), \qquad \xi\in D.
\]
Given $\xi\in D,$ we have  $\p{\Phi(\xi)}[\zeta]=(1+(\p_x\xi)^2)^{-1/2}\p_x^2\zeta+\bA\zeta$ with $\bA\in\mathcal{L}(h^{2+\alpha}(\s), h^{1+\alpha}(\s)).$ 
{Using well known interpolation properties of the small H\"older spaces:
\[
 (h^{\sigma_0}(\s), h^{\sigma_1}(\s))_\theta=h^{(1-\theta)\sigma_0+\theta\sigma_1}(\s)
\]
if $\theta \in (0,1)$ and $(1-\theta)\sigma_0+\theta \sigma_1\in \R^+ \setminus \N,$ where $(\cdot,\cdot)_\theta$  denotes the continuous interpolation method of DaPrato and Grisvard \cite{DG} ( see also \cite{A,L} ),
we obtain from Corollary 3.1.9,  Proposition 2.2.7,   and Proposition 2.4.1   in \cite{L} that the Fr\'echet derivative   $\p{\Phi(\xi)}$ 
is the generator of an analytic semigroup in $\mathcal{L}(h^\alpha(\s)).$}
Corollary 8.4.6 in \cite{L} {ensures} that the unique solution of \eqref{CP} to the initial data $\xi(0)=\eta$ is analytic, that is  $\xi\in C^\omega((0,\infty), h^{2+\alpha}(\s)),$ 
which implies the desired assertion.
\end{proof}

\section{Bifurcation analysis and the main result}\label{S:2}
In the remainder of this paper  we restrict our analysis to the stable regime when the fluid in $\0_\eta^b$ is more dense than that located above, that is $\h>\ov\h $,
and the fluids have constant  vorticity  
\begin{equation}\label{CV}
\omega=\gamma\in\R\qquad\text{and}\qquad\ov\omega=\ov\gamma\in\R.
 \end{equation}
Furthermore, the surface tension coefficient  may take any value $\sigma\in[0,\infty).$

As a first step, we introduce  a constant $\lambda$ into the problem which will allow us  to describe the   laminar flow solutions of \eqref{eq:PP}, i.e. solutions with a flat wave profile, located at  $y=0$.
Later on, we use this constant  as a bifurcation parameter to find  non-flat solutions of \eqref{eq:PP}.  
To this end,  when $\eta=0,$ we observe that the functions  $(\psi,\ov \psi)=:(\psi_0,\ov\psi_0)$ solving 
\eqref{eq:psiDP} are given by
\begin{equation}\label{lf}
\psi_0(x,y):=\frac{\omega y^2}{2} \ \ \text{in $\0_0^b$}\qquad\text{and}\qquad\ov\psi_0(x,y):=\frac{\ov \omega y^2}{2}+\lambda y\ \ \text{in $\0_0^t$},
\end{equation}
provided the constants $m$ and $\ov m$ satisfy
\begin{equation}\label{eq:conm}
m=\omega/2\qquad\text{and}\qquad \ov m=\lambda+\ov\omega/2.
\end{equation}
Hence, given $\lambda\in\R,$ the tuple $(\eta,\psi,\ov\psi):=(0,\psi_0,\ov\psi_0)$ is a  solutions of \eqref{eq:PP}
 if the constants $m,\ov m,$ and $Q$ are given by \eqref{eq:conm} and \eqref{eq:CR1}.

In order to determine  non-flat solutions of \eqref{eq:PP} we re-write problem \eqref{eq:PP} as a nonlinear and nonlocal operator equation having $(\lambda,\eta)$ as unknowns.
Because the domains where the Dirichlet problems \eqref{eq:psiDP} are posed depend upon $\eta$, we  need to transform these problems and the equation \eqref{eq:CR} on fixed reference manifolds.
Therefore, we define the functions
$\phi_{ \eta}:\Omega^b_0\to\Omega_\eta^b$  and $\ov\phi_{ \eta}:\Omega^t_0\to\Omega_\eta^t$ by the relations
\[
\phi_{\eta}(x,y):=(x,y+(1+y)\eta(x))\ \ \text{in $\0_0^b$}\qquad\text{and}\qquad\ov\phi_{\eta}(x,y):=(x,y+(1-y)\eta(x))\ \ \text{in $\0_0^t$},
\]
and observe that if $\eta$ belongs to $\U$ (see relation \eqref{U}), then $\phi_{ \eta}$  and $\ov\phi_{ \eta}$ are  diffeomorphisms.
We use  these diffeomorphisms to transform the Laplace operator into a differential operator on the fixed domains  $\0_0^b$ and $\0_0^t,$  respectively.
More precisely, setting 
\begin{align*}
&\A(\eta)w:=\big(\Delta(w\circ\phi_\eta^{-1})\big)\circ\phi_\eta,\ \ w\in C^{2+\alpha}\big(\ov{\0_0^b}\big),\\
&\ov\A(\eta){\ov w}:=\big(\Delta(\ov w\circ\ov\phi_\eta^{-1})\big)\circ\ov\phi_\eta,
\ \ \ov w\in C^{2+\alpha}\big(\ov{\0_0^t}\big),
\end{align*}
 we find the following  expressions for the differential operators $\A(\eta) $ and $\ov\A(\eta):$
\begin{align*}
\A(\eta)&=\p^2_{xx}-2\frac{(1+ y)\eta'}{1+ \eta}\p^2_{xy}+
\left(\frac{(1+ y)^2\eta'^2}{(1+ \eta)^2}+\frac{1}{(1+ \eta)^2}\right)\p^2_{yy}-(1+ y)\frac{(1+ \eta)\eta''- 2\eta'^2}{(1+ \eta)^2}\p_y,\\
\ov \A(\eta)&=\p^2_{xx}-2\frac{(1- y)\eta'}{1- \eta}\p^2_{xy}+
\left(\frac{(1- y)^2\eta'^2}{(1- \eta)^2}+\frac{1}{(1- \eta)^2}\right)\p^2_{yy}-(1- y)\frac{(1- \eta)\eta''+ 2\eta'^2}{(1- \eta)^2}\p_y
\end{align*}
for all $ \eta\in\U.$
>From these explicit expressions we can easily see  that the functions $\eta\mapsto\A(\eta)$ and
 $\eta\mapsto\ov\A(\eta)$ are both real-analytic.
Furthermore, corresponding to the coupling condition \eqref{eq:CR} we define the boundary operators
$\B:\U\times C^{2+\alpha}\big(\ov{\0_0^b}\big) \to C^{1+\alpha}(\s)$ and $\ov \B:\U\times C^{2+\alpha}\big(\ov{\0_0^t}\big) \to C^{1+\alpha}(\s)$ by the relations
\[
\B(\eta,w):=\tr \big(|\nabla (w\circ\phi_\eta^{-1})|^2\circ\phi_\eta\big)\qquad\text{and}\qquad \ov\B(\eta,w):=\tr\big( |\nabla (\ov w\circ\ov\phi_\eta^{-1})|^2\circ\ov\phi_\eta\big),
\]
with $\tr $ being the trace operator with respect to the line $\s\times\{0\}\cong \s.$
For these operators we find
\begin{align*}
\B(\eta,w)&:=\tr (\p_x w )^2-\frac{2\eta'}{1+\eta}\tr \p_x w \tr \p_y w +\frac{1+\eta'^2}{(1+\eta)^2}\tr( \p_y w)^2, \qquad (\eta,w)\in \U\times C^{2+\alpha}\big(\ov{\0_0^b}\big),\\
\ov \B(\eta,\ov w)&:=\tr (\p_x\ov w )^2-\frac{2\eta'}{1-\eta}\tr \p_x \ov w \tr \p_y\ov w +\frac{1+\eta'^2}{(1-\eta)^2}\tr( \p_y\ov w)^2,\qquad (\eta,\ov w)\in \U\times C^{2+\alpha}\big(\ov{\0_0^t}\big),
\end{align*}
which shows that $\B$ and $\ov B$ also depend   analytically on their variables.

To conclude, we observe that if $\psi$ and $\ov\psi$ are the solutions of the problems \eqref{eq:psiDP} (for some $\eta\in\U$) when   $\omega, \ov\omega$  are constant and the constants $m, \ov m$ are given by \eqref{eq:conm}, 
then $w:=\psi\circ\phi_\eta=:\T(\eta)$ and $\ov w:=\ov \psi\circ\ov\phi_\eta=:\ov\T(\lambda,\eta)$ are the solutions 
the solutions of the Dirichlet problems
\begin{equation}\label{eq:DPs}
\left\{
\begin{array}{rllll}
\A(\eta) w&=&\omega&\text{in}&\0^b_0,\\
w&=&0&\text{on}&y=0,\\
w&=&\omega/2&\text{on} &y=-1
\end{array}
\right.
\qquad
\text{and}
\qquad
\left\{
\begin{array}{rllll}
\ov \A(\eta) \ov w&=&\ov\omega&\text{in}&\0^t_0,\\
\ov w&=&0&\text{on}&y=0,\\
\ov w&=&\lambda+\ov \omega/2&\text{on} &y=1,
\end{array}
\right.
\end{equation}
respectively.
Since, $\A$ and $\ov\A$ are real-analytic, and the right-hand side of the equations \eqref{eq:DPs} depends analytically on $\lambda,$ we obtain that $\T$ and $\ov\T$ are also real-analytic.
{For the proof we refer to Lemmas 2.2 and 2.3. in \cite{ES}}.

With this notation, finding the solutions $(\eta,\psi,\ov\psi)$ of the coupled problem  \eqref{eq:PP} when the fluids have constant vorticities and the constants
$m, \ov m, Q$ are given by relations \eqref{eq:conm}, \eqref{eq:CR1}, reduces to determining the solutions  $(\lambda,\eta)$ of the nonlocal and nonlinear equation 
\begin{equation}\label{NE}
\Psi(\lambda,\eta)=0,
\end{equation}
where $\Psi:\R\times\U\to C^{\sigma'+\alpha}(\s)$  is given by
 \begin{align*}
\Psi(\lambda,\eta):=&\ov \h\ov \B(\eta, \ov\T(\lambda,\eta))-\h\B(\eta, \T(\eta))+2g(\ov \h-\h)\eta+2\sigma\eta''/(1+\eta'^2)^{3/2}\\
&-\int_\s(\ov\h \ov \B(\eta, \ov\T(\lambda,\eta))-\h\B(\eta, \T(\eta))+2 g (\ov \h- \h )\eta(x))\, dx,\qquad (\lambda,\eta)\in \R\times\U,
\end{align*}
with $\sigma'=0$ if $\sigma>0$ and $\sigma'=1$ for $\sigma=0.$

We note  that the laminar flow solutions found at the beginning of the section  are in correspondence with the trivial solutions $(\lambda,\eta)=(\lambda,0)$
of \eqref{NE}.
By applying the bifurcation theorem from simple eigenvalues due to Crandall and Rabinowitz \cite{CR} to equation \eqref{NE} 
we  show in Theorem \ref{MT:1} that infinitely many analytic branches consisting  of non-flat solutions of \eqref{eq:PP} 
intersect this trivial set of solutions of \eqref{NE}.
To this end,  we restrict first the domain and range of $\Psi$.
This is done by introducing the subspaces   $C^{m+\alpha}_{k,0,ev}(\s)$, $k, m\in\N,$  of $C^{m+\alpha}(\s)$ which contain only even, $2\pi/k$ periodic functions with integral mean zero.
Choosing functions with integral mean zero corresponds to a choice of equal volumes of both liquids beneath the rigid lid, meaning that,  at rest, the flat wave profile is located at $y=0.$ 
We use next the intrinsic properties of \eqref{eq:PP} to prove the following lemma.
\begin{lemma}\label{L:1} Let $\sigma\in [0,\infty)$ and  $\U_{k,0,ev}:=\U \cap C^{2+\alpha}_{k,0,ev}(\s).$ 
The operator $\Psi$ is real-analytic
\[
\Psi\in C^\omega(\R\times \U_{k,0,ev}, C^{\sigma'+\alpha}_{k,0,ev}(\s)),
\]
and 
\begin{align*}
\Psi(\lambda,\eta)=&\ov \h\ov \B(\eta, \ov\T(\lambda,\eta))-\h\B(\eta, \T(\eta))+2g(\ov \h-\h)\eta+\frac{2 \sigma\eta''}{(1+\eta'^2)^{3/2}}\\
& -\int_\s(\ov\h \ov \B(\eta, \ov\T(\lambda,\eta))-\h\B(\eta, \T(\eta)))\, dx
\end{align*}
for all $ (\lambda,\eta)\in \R\times \U_{k,0,ev}.$
\end{lemma}
\begin{proof}
The analyticity of $\Psi$ is  a consequence of the fact that the operators $\T,  \ov\T, \B, \ov \B$ are real-analytic in their variables.

 Let now $(\lambda,\eta)\in\R\times \U_{k,0,ev}$ be given. 
Since $\eta $ is even, we may define the functions $w_1(x,y):=w(-x,y)$ for $(x,y)\in\0_0^b$ and $\ov w_1(x,y):=\ov w(-x,y)$ for $(x,y)\in\0_0^t,$
where $w$ and $\ov w$ denote  the solutions of \eqref{eq:DPs}, respectively.
Since the Dirichlet conditions are constant, it may be easily checked that $w_1$ and $\ov w_1$ 
are solutions of the problems \eqref{eq:DPs}, respectively, so that, by the weak elliptic maximum principle,  $w=w_1$ and $\ov w=\ov w_1$.
 A similar argument shows that $w$ and $\ov w$ are both $2\pi/k$ periodic in $x$ and  we obtain, by using the explicit relations for the boundary operators $\B$ and $\ov \B$, that $\Psi(\lambda,\eta)$ is $2\pi/k$ periodic and even. 
To finish the proof, we note, with our choice of $Q$ in \eqref{eq:CR1},  that  $\Psi(\lambda,\eta)$ has integral mean zero if $\eta$ has this property. 
\end{proof}
 
An important step  in our analysis is to  determine the Fr\'echet derivative of $\Psi $ at the trivial branch of laminar solutions.

\begin{lemma}\label{L:2}
Given $\lambda\in\R,$ the Fr\'echet derivative $\p_\eta\Psi(\lambda,0)\in\mathcal{L}(C^{2+\alpha}_{k,0,ev}(\s), C^{\sigma'+\alpha}_{k,0,ev}(\s))$ is a Fourier multiplier
\begin{equation}\label{FD}
\p_\eta\Psi(\lambda,0)\sum_{k=1}^\infty a_k\cos(kx)=\sum_{k=1}^\infty\mu_k(\lambda) a_k\cos(kx)
\end{equation}
with symbol 
\begin{align}\label{FDS}
\mu_k(\lambda)&:=2\left[g(\ov\h-\h)-\sigma k^2+\ov \omega\hspace{0.5mm}\ov\h\lambda+\ov\h\frac{k}{\tanh(k)}\lambda^2\right],\qquad k\geq1.
\end{align}
\end{lemma}
\begin{proof} Taking into account that $\T(0)=\psi_0$ and $\ov\T(\lambda,0)=\ov\psi_0$ for all $\lambda\in\R$, we  obtain from
\begin{equation}\label{firstorder}
\begin{array}{lllllll}
&\p_\eta\B(0,\psi_0)=0,  \qquad &\p_w\B(0,\psi_0)=0,\\
&\p_\eta\ov\B(0,\ov\psi_0)[\eta]=2\lambda^2\eta,   \qquad &\p_{\ov w}\ov \B(0,\ov\psi_0)[\ov w]=2\lambda \tr \p_y \ov w, 
\end{array}
\end{equation}
and the chain rule that
\begin{align}\label{eq:FDs}
\p_\eta\Psi(\lambda,0)[\eta]=&2\ov\h\lambda\left( \tr( \p_y \p_\eta\ov\T(\lambda,0)[\eta])-\int_\s\tr( \p_y \p_\eta\ov\T(\lambda,0)[\eta])\, dx \right)\\
&+[2g(\ov\h-\h)+2\lambda^2\ov \h]\eta+2\sigma\eta''
\end{align}
for all $\eta\in C^{2+\alpha}_{k,0,ev}(\s).$
We determine now the Fr\'echet  derivative $\ov w:=\p_\eta\ov\T(\lambda,0)[\eta].$
Therefore, we differentiate  the equations of the Dirichlet problem \eqref{eq:DPs} solved by  $\ov\T(\lambda,\eta)$ with respect to $\eta$  at $\eta=0$,  
and find that $\ov w$ is the solution of the problem
\begin{equation}\label{eq:LDP}
\Delta \ov w=-\p\ov \A(0)[\eta]\ov\psi_0 \quad \text{in}\ \0^t_0,\qquad \ov w=0 \quad \text{on}\ \p\0^t_0,
\end{equation}
where, by the explicit expressions found at the beginning of the section
\[
\p\ov\A(0)[\eta]\ov\psi_0=2\ov\omega \eta-(1-y)(\ov\omega y+\lambda)\eta''.
\]
In order to determine  relation \eqref{FD} we use Fourier expansions for functions in  $C^{2+\alpha}_{k,0,ev}(\s)$ and make a similar ansatz for the solution $\ov w:=\p_\eta\ov\T(\lambda,0)[\eta]$ of \eqref{eq:LDP}:
\[
\eta=\sum_{k=1}^\infty a_k\cos(kx)\quad\text{and}\quad \ov w=\sum_{k=1}^\infty a_k \ov w_k(y)\cos(kx).
\]
The right-hand side of the first equation of \eqref{eq:LDP} may be then expanded as follows
\[
-\p\ov\A(0)[\eta]\psi_\ap=\sum_{k=1}^\infty a_k b_k(y)\cos(kx)\qquad\text{with}\qquad b_k(y):=-2\ov\omega-(1-y)(\ov\omega y+\lambda)k^2,
\]
and, plugging all these expansions into \eqref{eq:LDP} and matching the coefficients corresponding to $\cos(kx),$ yields that $\ov w_k$ solves
\begin{equation}\label{eq:oe}
\ov w_k''-k^2\ov w_k=b_k,\quad 0<y<1,\qquad \text{and}\qquad\ov w_k(0)=\ov w_k(1)=0.
\end{equation}
The solution of \eqref{eq:oe} is given by
\begin{align}\label{w_k}
\ov w_k(y)=\frac{\lambda}{\tanh(k)}\sinh(ky)-\lambda\cosh(ky)+\ov\omega(y-y^2)+\lambda(1-y),
\end{align}
and together with  \eqref{eq:FDs} we obtain the desired relations \eqref{FD} and \eqref{FDS}.
\end{proof}

We recall now a theorem which provides a sufficient condition for an operator in order to be a Fourier multipliers between H\"older spaces of  periodic functions. 
\begin{lemma}[{\cite[Theorem 3.4]{EM}} ] \label{UT}
Let $r,s$ be two positive {non-integer} constants and let $(M_p)_{p\in\Z}\subset\mathbb{C}$ be a sequence satisfying the following conditions
\begin{align*}
&(i)\quad\underset{p\in\Z\setminus\{0\}}\sup\,|p|^{r-s}|M_p|<\infty,\\[1ex]
&(ii)\quad\underset{p\in\Z\setminus\{0\}}\sup\,|p|^{r-s+1}|M_{p+1}-M_p|<\infty,\\[1ex]
&(iii)\quad\underset{p\in\Z\setminus\{0\}}\sup\,|p|^{r-s+2}|M_{p+2}-2M_{p+1}+M_p|<\infty.
\end{align*}
Then, the Fourier multiplier satisfies
\[\sum_{p\in\Z}a_pe^{ipx}\longmapsto \sum_{p\in\Z}M_pa_pe^{ipx}\in\mathcal{L}(C^s(\s),C^r(\s))\].
\end{lemma}
Using Lemma \ref{UT} and the relations \eqref{FD} and \eqref{FDS}, it is not difficult to see that   $\p_\eta\Psi(\lambda,0)$ is an isomorphism if  $\lambda$ is chosen such that 
$\mu_k(\lambda)\neq0$ for all integers $ k\geq1.$
To find these values of  $\lambda$  we solve  the quadratic equation $\mu_k(\lambda)=0$ and find the solutions
\[
{\Lambda^{i}_k}:=-\frac{\ov \omega}{2}\frac{\tanh(k)}{k}{+(-1)^i}\sqrt{\frac{g(\h-\ov \h)+\sigma k^2}{\ov \h}\frac{\tanh(k)}{k}+\frac{\ov\omega^2}{4}\frac{\tanh^2(k)}{k^2}}, \quad i=1,2.
\]
We consider the case $\sigma=0$ first.
Since $x\mapsto\tanh(x)/x$ is a strictly decreasing function mapping  $[0,\infty)$ onto $(0,1]$, and, for $B>0,$
 $t\mapsto t+\sqrt{t^2+Bt}$ is increasing (resp. $t\mapsto t-\sqrt{t^2+Bt}$ is decreasing) on $[0,1]$, 
we conclude that $(\Lambda^i_k)$ are strictly monotone sequences.
If $\sigma\neq0,$ we choose $\sigma_0>0$ such that $(\Lambda^i_k)$ are strictly monotone sequences for all $\sigma>\sigma_0.$
By Lemma \ref{UT}, we see that in both cases $\p_\eta\Psi(\Lambda^i_k,0),  i=1,2,$ is a Fredholm operator of index zero having a one dimensional kernel.

This leads us to the existence result of this paper.
Besides existence of analytic curves consisting of traveling internal waves, we determine the second order Taylor expansions for these curves in a neighbourhood of the laminar flow solutions.
These expansions give us sufficient information to provide the precise picture of the streamlines (the level curves of the stream functions) in the frame  moving with wave speed $c$, see Theorem \ref{MT:2}.

\begin{thm} \label{MT:1} Assume that $\ov \h<\h$ and if $\sigma\neq 0$, then let $\sigma>\sigma_0$. 
Given $k\in\N,$ $k\geq1,$ there exists  $\e_k>0$ and real-analytic curves 
\[(\lambda_k^{i},\eta_k^i):(-\e_k,\e_k)\to\R\times C^{2+\alpha}_{k,0,ev }(\s), \qquad i=1,2,\]
consisting only of real-analytic solutions of \eqref{NE} of minimal period $2\pi/k$  and having exactly one crest and trough per period. 
These are the only solutions of \eqref{NE} close to  $(\Lambda_k^i,0)$, and for  $s\to0$ we have
\begin{equation}\label{eq:asy}
\begin{aligned}
\lambda_k^{i}(s)&=\Lambda_k^{i}+O(s^2),\\
\eta_k^i (s)&=-\cos(kx)s+\alpha_k^i \cos(2kx)s^2+O(s^3)\quad\text{ in} \ C^{2+\alpha}_{k,0,ev }(\s),
\end{aligned} \qquad i=1,2,
\end{equation}
with constants $\alpha_k^i$ given by \eqref{eq:const} (with $\Lambda$ replaced by $\Lambda_k^i$).
\end{thm} 
\begin{proof} We verify first  that the assumptions of the theorem on bifurcations from simple eigenvalues due to Crandall and Rabinowitz \cite[Theorem 1.7]{CR} are satisfied.
We already know that, when $\lambda=\Lambda_k ^i$ for some $k\in\N$ and $i=1,2,$ the Fr\'echet   $\p_\eta\Psi(\Lambda^i_k,0)$ is a Fredholm operator with 
\begin{equation*}
\Kern \p_\eta\Psi(\Lambda^i_k,0)=\spn\{\cos(kx)\} \qquad\text{and}\qquad \im \p_\eta\Psi(\Lambda^i_k,0)\oplus \spn\{\cos(kx)\}=C_{k,0, ev}^{\sigma'+\alpha}(\s).
\end{equation*}
Furthermore, differentiating \eqref{FD} with respect to $\lambda$ we obtain that
\[
\p^2_{\lambda\eta}\Psi(\Lambda^i_k,0)[\cos(kx)]==\pm4\ov \h\sqrt{\frac{g(\h-\ov \h)+\sigma k^2}{\ov \h}\frac{k}{\tanh(k)}+
\frac{\ov\omega^2}{4}}\cos(kx),
\]
which implies $\p^2_{\lambda\eta}\Psi(\Lambda^i_k,0)[\cos(kx)]\notin\im \p_\eta\Psi(\Lambda^i_k,0).$
The existence of the analytic bifurcation curves follows now from the above mentioned theorem and Lemma \ref{L:1}.

We pick now a solution $(\lambda(s),\eta(s))$ of \eqref{NE} located on one of the curves $(\lambda_k^{i},\eta_k^i), \ i=1,2,$ and denote by
 $\psi:=\T(\eta(s))\circ\phi_{\eta(s)}^{-1}$ and $\ov \psi:=\ov\T(\lambda, \eta(s))\circ\ov\phi_{\eta(s)}^{-1}$ the  stream functions  associated to it.
In order to prove that $\eta(s)$ is real-analytic, we show  that the assumption \eqref{eq:CD} is fulfilled provided $\e_k$ is sufficiently small.
 Indeed, since $\T(0)=\psi_0$ and $\ov\T(\Lambda_k^i,0)=\ov\psi_0,$ with $\psi_0$ and $\ov\psi_0$ given by \eqref{lf}, we obtain  from 
 $|\nabla \psi_0|^2+| \nabla \ov\psi_0|^2=|\Lambda_k^i|^2>0$ on $y=0$ that \eqref{eq:CD} is satisfied by the laminar flows $(\lambda_k^{i}(0),\eta_k^i(0)), \ i=1,2.$
Choosing $\e_k$ small enough, the real-analyticity of $\eta(s)$ follows now from Theorem \ref{MT:0}, by making use of the continuity  of the bifurcation curves and of the solutions operators $\T$  and $\ov\T$.

Further on,  we prove the asymptotic expansions  \eqref{eq:asy} and show that the internal traveling waves we obtain have exactly one crest and trough per period.
To this end, we fix $k\in\N$, $k\geq 1,$ and, to ease notation, we  let in this final part of the proof $\Lambda:=\Lambda_k^{i}$, for 
 $i\in\{1,2\},$ and denote by $(\lambda,\eta)$ the corresponding branch of solutions $(\lambda_k^{i},\eta_k^i).$
Since the bifurcation curve $(\lambda,\eta)$ is real-analytic, we obtain from Theorem 1.7 in \cite{CR} that 
\begin{equation}\label{eq:asy2}
\lambda(s)=\Lambda+{\lambda_s(0)}s+O(s^2)\qquad\text{and}\qquad\eta (s)=-s\cos(kx)+\tau(s)+O(s^2)\quad\text{ in} \ C^{2+\alpha}_{k,0, ev }(\s),
\end{equation}
with $\tau(0)={\tau_s(0)}=0$ {and where the index $s$ denotes the derivative with respect to the variable $s$.}
Additionally, the function $\tau$ takes values in the closed complement $X_0 $ of $\spn\{\cos(kx)\} $ in $C^{2+\alpha}_{k,0,ev}(\s).$
Proceeding similarly as in  \cite{W}, we find  first from \eqref{eq:asy2} that
 \[
 \eta'=sk\sin(kx)+O(s^2)\qquad\text{and}\qquad \eta''={s}k^2\cos(kx)+{O(s^2)}\qquad\text{ in} \ C(\s).
 \]
 Therefore,  $\eta''(0)>0$, $\eta''(\pi/k)<0,$
 and, since $\eta'$ is odd, we also have $\eta'(0)=\eta'(\pi/k)=0.$ 
We resume that $\eta'$ is positive on $(0,\pi/k)$ provided $\e_k$ is small, meaning that the wave has its  creast  located at $x=\pi/k$ and the trough at $x=0.$

The next step of the proof is to determine  the derivatives ${\lambda_s(0)}$ and ${\tau_{ss}(0)}.$
Differentiating the relation $\Psi(\lambda(s),\eta(s))=0$ twice with respect to $s$ we find,  at $s=0$,  that   ${\lambda_s(0)}$ and  ${\tau_{ss}(0)}$ are related by 
\begin{equation}\label{eq:firstequation}
-2{\lambda_s(0)}\p^2_{\lambda,\eta}\Psi(\Lambda,0)[\cos(kx)]+\p^2_{\eta\eta}\Psi(\Lambda,0)[\cos(kx)]^2  +\p_\eta\Psi(\Lambda,0)[{\tau_{ss}(0)}]=0.
\end{equation}
Clearly, we need to find the second order derivative 
\begin{equation}\label{2psi}
\begin{aligned}
\p^2_{\eta\eta}\Psi(\lambda,0)[\cos(kx)]^2=&\ov \h\left(\ov I-\int_\s\ov I\, dx\right)- \h\left( I-\int_\s I\, dx\right),
\end{aligned}
\end{equation}
where, setting  $\xi:=\cos(kx),$ we made the following notation
\begin{equation}\label{2psi2}
\begin{aligned}
I:=&\p^2_{\eta\eta}\B(0,\psi_0)[\xi]^2-2\p^2_{\eta w}\B(0,\psi_0)[\xi,\p\T(0)[\xi]]-\p^2_{w w}\B(0,\psi_0)[\p\T(0)[\xi]]^2\\
&-\p_{ w}\B(0,\psi_0)[\p^2\T(0)[\xi]^2], \\
\ov I:=&\p^2_{\eta\eta}\ov\B(0,\ov\psi_0)[\xi]^2+2 \p^2_{\eta\ov w}\ov\B(0,\ov\psi_0)[\xi,\p_\eta\ov\T(\Lambda,0)[\xi]]
+\p^2_{\ov w\hspace{0.25mm} \ov w}\B(0,\ov\psi_0)[\p_\eta\ov\T(\Lambda,0)[\xi]]^2\\
&+ \p_{\ov w}\ov\B(0,\ov\psi_0)[\p^2_{\eta\eta}\ov\T(\Lambda,0)[\xi]^2].
\end{aligned}
\end{equation}
In view of  \eqref{lf}, we compute
 \[
\begin{array}{lll}
&\p^2_{\eta\eta} \ov\B(0,\ov\psi_0)[\xi]^2=2\Lambda^2\left(3\xi^2+\xi'^2\right), & \p^2_{\eta \ov w}\ov \B(0,\ov\psi_0)[\xi,\ov v]=4\Lambda\xi\tr\p_y\ov v-2 \xi'\Lambda \tr \p_x \ov v,\\[1ex]
&\p^2_{ \ov w\hspace{0.25mm}\ov w}\B(\eta,w)[\ov v]^2=2\tr\left((\p_x \ov v)^2+(\p_y\ov v)^2\right),& \p^2_{ ww}\B(\eta,w)[v]^2=2\tr \left((\p_x v)^2+(\p_y v)^2\right),\\[1ex]
&\p^2_{\eta\eta} \B(0,\psi_0)=\p^2_{\eta w}\B(0,\psi_0)=0, 
\end{array}
\]
and, recalling  \eqref{firstorder}, we have $\p_{ w}\B(0,\psi_0)=0.$ 
Consequently, we need   to determine only the derivatives $\p\T(0)[\xi]$ and
 $\p^2_{\eta\eta}\ov\T(\Lambda,0)[\xi]^2$ to obtain an explicit expression for the right-hand side of \eqref{2psi}.
Concerning  $\p\T(0)[\xi]$, we have to study a linear Dirichlet problem similar to \eqref{eq:LDP}, and one finds 
\begin{equation}\label{For1}
\p\T(0)[\cos(kx)]=\omega(y^2+y)\cos(kx).   
\end{equation}
As for the  the second order derivative, we differentiate the Dirichlet problem \eqref{eq:DPs} for $\ov w$ twice with respect to $\eta$ and  obtain, at $\eta=0$, that
$\ov v:=\p^2_{\eta\eta}\ov\T(\Lambda,0)[\xi]^2$ is the solution of the problem
\begin{equation}\label{eq:2DPs2}
\Delta \ov v=-2\p\ov \A(0)[\xi]\p_\eta\ov\T(\Lambda,0)[\xi]-\p^2\ov \A(0)[\xi,\xi]\ov\psi_0\quad \text{in} \, \0^t_0,\qquad \ov v=0 \quad \text{on} \, \p\0_0^t,
\end{equation}
where
\begin{align*}
&\p^2\ov \A(0)[\xi]^2\ov\psi_0=\ov\omega\left(2(1-y)^2\xi'^2+6\xi^2\right)-(1-y)(\ov\omega y+\Lambda)\left(4\xi'^2+2\xi\xi''\right),\\
&\p\ov \A(0)[\xi]\p_\eta\ov\T(\Lambda,0)[\xi]=\left(-2(1-y)\xi'\p_{xy}^2+2\xi\p_{yy}^2-(1-y)\xi''\p_y\right)\left(\p_\eta\ov\T(\Lambda,0)[\xi]\right).
\end{align*}
By \eqref{w_k}, we have   $\p_\eta\ov\T(\Lambda,0)[\cos(kx)]=w_k(y)\cos(kx),$ so that  we can express the right-hand side  of the first equation of \eqref{eq:2DPs2}
as follows
\begin{align*}
&2\p\ov \A(0)[\xi]\p_\eta\ov\T(\Lambda,0)[\xi]+\p^2\ov \A(0)[\xi]^2\ov\psi_0=E_0(y)+E_{2k}(y)\cos(2kx),
\end{align*}
whereby 
\begin{align*}
E_0(y):=&-\ov \omega+2k^2\Lambda\left(\frac{\sinh(ky)}{\tanh(k)}-\cosh(ky)\right)-(1-y)k^3\Lambda\left( \frac{\cosh(ky)}{\tanh(k)}-\sinh(ky)\right)\\
E_{2k}(y):=&-\ov\omega+2k^2\ov \omega(1-y)^2+2k^2\Lambda\left(\frac{\sinh(ky)}{\tanh(k)}-\cosh(ky)\right)\\
&+3(1-y)k^3\Lambda\left( \frac{\cosh(ky)}{\tanh(k)}-\sinh(ky)\right).
\end{align*}
Due to the linearity of \eqref{eq:2DPs2}, we write  
 $\p^2_{\eta\eta}\ov \T(\Lambda,0)[\cos(kx)]^2=c_k(y)+d_k(y)\cos(2kx),$    with
 $c_k$  denoting the solution of \eqref{eq:oe} when $(k,b_k)=(0,-E_0)$  and $d_k$ being the solution of \eqref{eq:oe} with $(k,b_k)$ are replaced by $(2k,-E_{2k}).$
Since all Fr\'echet derivatives of $\T$ and $\ov\T$ with respect to $\eta$ have zero boundary values, of relevance for our purpose are only the first derivatives 
\begin{align*}
&c'_k(0):=-k^2\Lambda-\frac{\ov\omega}{2},\\[1ex]
&d'_k(0):=-k^2\Lambda -\frac{k\Lambda }{\tanh(k)}+\frac{2k^2\Lambda }{\tanh(k)\tanh(2k)}-\ov\omega+\frac{k\ov\omega}{\tanh(2k)}.
\end{align*}
Summarising, we have shown that
\begin{align*}
I=&\omega^2+\omega^2\cos(2kx),\\
\ov I=&\left(\frac{k\Lambda }{\tanh(k)}+\ov\omega\right)^2+\frac{2k\Lambda^2}{\tanh(k)}+\ov\omega^2+\Lambda\ov\omega-\Lambda^2k^2\\
&+\left[\left(\frac{k\Lambda }{\tanh(k)}+\ov\omega\right)^2+\frac{4k^2\Lambda^2 }{\tanh(k)\tanh(2k)}+\frac{2k\ov\omega\Lambda }{\tanh(2k)}-3k^2\Lambda^2\right]\cos(2kx).
\end{align*}
Recalling \eqref{2psi}, we  find that $\p^2_{\eta\eta}\Psi(\Lambda,0)[\cos(kx)]^2=A_k\cos(2kx),$
whereby 
\begin{equation}\label{eq:constants}
A_k:=\left[\left(\frac{k\Lambda }{\tanh(k)}+\ov\omega\right)^2+\frac{4k^2\Lambda^2 }{\tanh(k)\tanh(2k)}+\frac{2k\ov\omega\Lambda }{\tanh(2k)}-3k^2\Lambda^2\right]-\h\omega^2.
\end{equation}

To finish, we observe that  $\p^2_{\eta\eta}\Psi(\Lambda,0)[\cos(kx)]^2$ belongs to  $\im \p_\eta\Psi(\Lambda,0)$, meaning that $\cos(kx)$ is orthogonal on $\p^2_{\eta\eta}\Psi(\Lambda,0)[\cos(kx)]^2$ in $L_2(\s).$
Whence, if we multiply the relation \eqref{eq:firstequation} 
by $\cos(kx)$ and integrate it then over the unit circle, we get that ${\lambda_s(0)}=0.$
Moreover, since the restriction $\p_\eta\Psi(\Lambda,0):X_0\to\im\p_\eta\Psi(\Lambda,0)$ is an isomorphism  and ${\tau_{ss}(0)}\in X_0,$  \eqref{eq:firstequation} yields
 \begin{equation}\label{eq:sequation}
{\tau_{ss}(0)}=-(\p_\eta\Psi(\Lambda,0))^{-1}\p^2_{\eta\eta}\Psi(\Lambda,0)[\cos(kx)]^2.
\end{equation}
Setting 
\begin{equation}\label{eq:const}
\alpha_k:=-\frac{A_k}{2\mu_{2k}(\Lambda)},
\end{equation}
we conclude from \eqref{FD}  that $\tau_{ss}(0)=2\alpha_k\cos(2kx),$
 and together with \eqref{eq:asy2} we find the desired expansion for $\eta$.
 This completes the proof.
\end{proof}

\section{Streamlines for internal waves with stagnation points on the profile}\label{S:3}
In the remainder of this paper we restrict our considerations to the  curve $(\lambda_k^{1},\eta_k^1)$ and denote  by $(\lambda,\eta)$ a solution of \eqref{eq:PP} which lies on this curve. 
As in the proof of Theorem \ref{MT:1}, we write $\psi$ and $\ov\psi$ for the stream functions associated to this solution.
The next theorem shows that the traveling  wave solutions found in Theorem \ref{MT:1} possess points  which are stagnation points when considering the wave profile as a part of the fluid located below, 
but loose this feature when considering the interface as being a part of  the fluid located above.
Particularly, as we approach  these points, the  velocity $(u,v)$ of the fluid particles located  below the wave profile satisfies $(u,v)\to(c,0)$.

We shall exemplify this in Theorem \ref{MT:2} in the particular case when $\ov\omega\geq0$ and  $\omega>0.$  
We may choose also $\omega<0, $ and only the orientation along   the streamlines beneath the wave profile has to be changed in Figure \ref{F:1}.
Allowing $\ov\omega$ to be negative, we may obtain stagnation points within the fluid  located above.
However, the picture of the additional critical layer we could obtain in $\0_\eta^t$ has been studied in detail in  \cite{CV2, EMM2, W}.

\begin{thm}[The streamlines in the moving frame]\label{MT:2} Additionally to Theorem \ref{MT:1}, assume that $\ov\omega\geq0$ and  $\omega>0.$ 
Then, within a period, the streamlines corresponding to a solution $(\lambda,\eta)$ on the curve $(\lambda_k^{1},\eta_k^1)$ are qualitatively described in Figure \ref{F:1}.
Particularly, there exists a critical layer consisting of closed streamlines which is delimited from above by the wave profile and from below by a separatrix which connects two stagnation points (solutions of $\nabla\psi=0$)
which are located on the wave profile.
Furthermore, there exists exactly one more stagnation point which is situated in the center of the critical layer.  
\end{thm}
  \begin{figure}
$$\includegraphics[width=10cm]{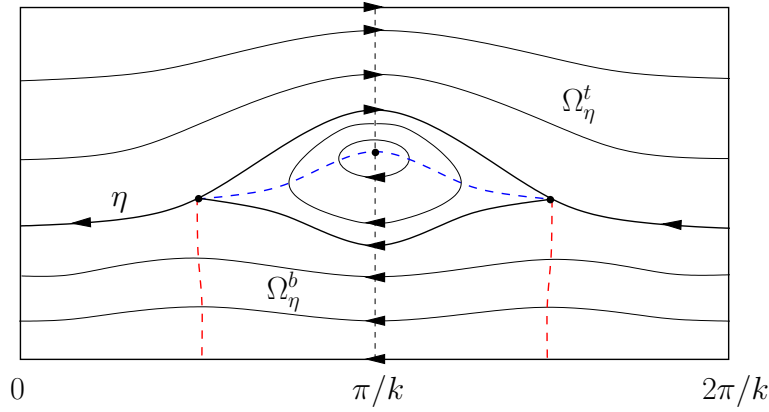}$$
\caption{This figure illustrates the streamlines in the moving frame in the case when $\omega>0,$ $\ov\omega\geq0$, and $\Lambda=\Lambda_k^1$. 
The dashed blue curve is the critical level where $\psi_y$ vanishes and the dashed red curves consist only of points where $\psi_x$ vanishes.
The critical layer of closed streamlines is located between the  wave profile and the separatrix (the thick line below the wave surface) which connects the two stagnation points on the wave profile.
A further stagnation point   is located below the wave profile on the line $x=\pi/k$.}
\label{F:1}
\end{figure}

In order to prove Theorem \ref{MT:2}, we rely on the expansions \eqref{eq:asy}.
Compared to \cite{EMM2, W} the situation we consider is more degenerate, because, by our choice  \eqref{eq:conm} of the constants,  bifurcation occurs when $\psi_{0,y}=0$ on $y=0,$ and we need  second order 
expansions for the bifurcation curves, cf. \eqref{eq:asy}, to be able to analyse the streamlines within the fluid domains. 
First, we prove:

\begin{prop}\label{P:2} Let   $(\lambda,\eta)$ be located on  the bifurcation curve $(\lambda_k^{1},\eta_k^1)$, and define
\[
\0_{b}:=\{(x,y)\in\0_\eta^b\,:\, 0<x<\pi/k\}.
\]
If $\e_k$ is small and $\omega\neq0$, we have the following properties:
\begin{itemize}
\item[$(a)$] If $\ov\omega\geq0$, then $\ov\psi_y>0$ in $\ov\0^{t}_\eta $ and $\ov\psi_x(x,y)<0$ for all $x\in(0,\pi/k)$ and $\eta(x)\leq y<1; $
\item[$(b)$] $\omega\psi_{yy}>0$ in $\ov{\0_\eta^b}$;
\item[$(c)$] There exists a unique point $\zeta\in(0,\pi/k)$ with $\zeta=\pi/(2k)+O(s)$, and a continuous curve $y_\zeta:[0,2\pi/k]\to \R$ which satisfies
\begin{itemize}
\item[$(i)$] $y_\zeta(x)=\eta(x)$ for $x\notin(\zeta,2\pi/k-\zeta)$ and $-1<y_\zeta(x)<\eta(x)$ for all $x\in(\zeta,2\pi/k-\zeta);$
\item[$(ii)$] $\omega\psi_y>0$ for $x\in (\zeta,2\pi/k-\zeta)$ and $y\in (y_\xi(x),\eta(x)]$,  $\psi_y(x,y_\zeta(x))=0$ for all $x\in[\zeta,2\pi/k-\zeta],$
and $\omega\psi_y<0$ elsewhere in $\0_\eta^b$; 
\end{itemize}
\item[$(d)$]  There exists  a curve $x=\xi(y):[-1,y_b]\to (0,\pi/k)$ such that  $(\xi(y_b),y_b)$ is a point on the wave profile and 
 \begin{equation}\label{eq:psi1}
 \left\{
 \begin{array}{lll}
 \omega\psi_{xy}(x,y)>0, &\text{if $y<y_b$ and $x<\xi(y)$};\\
  \omega\psi_{xy}(x,y)=0, &\text{if $x=\xi(y)$};\\
  \omega\psi_{xy}(x,y)<0,&\text{if $y_b<y$ or $[y<y_b$ and $\xi(y)<x<\pi/k]$}.
 \end{array}
 \right.
 \end{equation}
 Moreover, $\xi=\pi/(2k)+O(s)$ uniformly in $y\in[-1,y_b].$ 
 \item[$(e)$] There exists  a curve $x=\ov \xi(y):[-1,\ov y_b]\to (0,\pi/k)$ such that  $(\ov\xi(\ov y_b),\ov y_b)$ is a point on the wave profile and 
 \begin{equation}\label{eq:psi2}
 \left\{
 \begin{array}{lll}
 \omega\psi_x(x,y)>0, &\text{if $y<\ov y_b$ and $x<\ov \xi(y)$};\\
 \omega\psi_x(x,y)=0, &\text{if $x=\ov \xi(y)$};\\
 \omega\psi_x(x,y)<0,&\text{if $y>\ov y_b$ or $[ y<\ov y_b$ and $\ov \xi(\ov y)<x<\pi/k]$}.
 \end{array}
 \right.
 \end{equation}
 Moreover, $(\zeta,y_\zeta(\zeta)=(\ov\xi(\ov y_b),\ov y_b)$ and $\ov\xi= \pi/(2k)+O(s)$ uniformly in $y\in[-1,\ov y_b].$
\end{itemize}
\end{prop}
 \begin{proof}  
 To prove the claim $(a)$, we  recall that $\ov\psi_{0,y}=2\ov\omega y+\Lambda_k^{1}>0$ in  $\ov\0^{t}_0$.  
Using  the continuity of  the operator $\ov\T$ and of the bifurcation curve $(\lambda_k^{1},\eta_k^1)$ it follows that  $\ov\psi_y>0$ in $\ov\0^{t}_\eta $.
 Furthermore, we notice that $\ov\psi_x(x,\eta(x))=-\eta'(x)\ov\psi_y(x,\eta(x))<0$ for all $x\in(0,\pi/k),$ and since $\ov\psi$ is even in $x$ we find that $\ov\psi_x\leq0$ on the boundary $\p\0_t,$
 whereby $\0_t:=\{(x,y)\in\0_\eta^t\,:\, x\in (0,\pi/k)\}.$ 
 Recalling \eqref{eq:psiDP}$_2$, we find that $\Delta\ov\psi_x=0$ in $\0_t$ and the claim $(a)$ follows from the strong elliptic maximum principle.

We consider  in the remainder of the proof only the fluid located below the interface $\eta.$  
The functions $\eta$ and $\psi$ being both even functions, it suffices to restrict our considerations to the domain $\0_b.$
In order to prove $(b)$ and $(c)$, we use the first order  Taylor expansion 
\[
w:=\T(\eta)=\T(0)+\p\T(0)[\eta]+O(s^2)=\frac{\omega y^2}{2}-s\p\T(0)[\cos(kx)]+O(s^2) \qquad \text{in} \ C^{2+\alpha}(\ov{\0_0^b}).
\]
Recalling \eqref{For1},    we get 
\begin{equation}\label{eq:axy}
w=\frac{\omega y^2}{2}-s\omega (y^2+y)\cos(kx)+O(s^2)\qquad \text{in} \ C^{2+\alpha}(\ov{\0_0^b}).
\end{equation}
Since $\psi(x,y)=w(x,(y-\eta(x))/(1+\eta(x)))$ for all $(x,y)\in\ov\0_\eta^b$, we immediately obtain the assertion $(b)$.
Furthermore, at $y=-1$, we have that $\p_y w(\cdot,-1)=-\omega+O(s)$, while differentiating \eqref{eq:axy} at $y=0$, yields
\[
\p_yw (x,0)=-s\omega\cos(kx)+O(s^2).
\]
Repeating the arguments in the proof of Theorem \ref{MT:1}, we find that $\omega\p_yw$ is strictly increasing on $[0,\pi/k]$.
Since $\omega\p_yw(0,0)<0$ and $\omega\p_yw(\pi/k,0)>0$, we deduce that there exists a unique $\zeta\in(0,\pi/k)$ such that $\p_yw(\zeta,0)=0.$
By virtue of $(b),$ we find for each $x\in(\zeta,\pi/k]$ a point $ \wt  y_\zeta(x)\in(-1,0)$ with the property that  $\omega\p_yw(x,y)<0$  for $y<\wt y_\zeta(x)$ and 
$\omega\p_yw(x,y)>0$ for $y>\wt y_\zeta(x).$
The mean value theorem shows that $\zeta$ is close to $\pi/2k$ in the sense that $\zeta=\pi/(2k)+O(s).$
We appeal now to the fact that  $\p_y w$ is even to obtain the desired conclusion $(c)$.

In order to prove $(d)$ and $(e)$, we are obliged, to use a better approximation for $\T(\eta)$ 
\begin{equation}\label{axy}
\begin{aligned}
w=&\T(\eta)=\T(0)+\p\T(0)[\eta]+\frac{1}{2}\p^2\T(0)[\eta]^2+O(s^3)\\
=&\frac{\omega y^2}{2}-s\p\T(0)[\cos(kx)]+s^2\left( \p\T(0)[\alpha_k\cos(2kx)]+\frac{1}{2}\p^2\T(0)[\cos(kx)]^2\right)+O(s^3) 
\end{aligned}
\end{equation}
in $C^{2+\alpha}(\ov{\0_0^b}),$  cf. \eqref{eq:asy}.
To determine $\p^2\T(0)[\cos(kx)]^2$ we differentiate the Dirichlet problem \eqref{eq:DPs} corresponding to $\0_\eta^b$ twice with respect to $\eta$, and find that $\p^2\T(0)[\cos(kx)]^2$  
is the solution of the Dirichlet problem
\begin{equation*}
\Delta v=-2\p\A(0)[\cos(x)]\p\T(0)[\cos(x)]-\p^2\A(0)[\cos(kx)]^2\psi_0\qquad\text{in} \, \0^b_0,\qquad  v=0\quad\text{on} \, \p\0_0^b.
\end{equation*}
Similarly as in the proof of Theorem \ref{MT:1}, we find that
\begin{equation}\label{H2}
\p^2\T(0)[\cos(kx)]^2=\frac{\omega(y^2+y)}{2}+\omega\beta_k(y)\cos(2kx),
\end{equation}  
whereby
\begin{equation}\label{H3}
\beta_k(y)=-\frac{1}{2}\left(\frac{\sinh(2ky)}{\tanh(2k)}+\cosh(2ky)-(1+y)^2\right).
\end{equation} 
For the clarity of the exposition we leave the proof of \eqref{H2}  to  the interested reader. 

We sum now the relations \eqref{axy}-\eqref{H3} and use \eqref{For1} to conclude  that
\begin{align}\label{pam}
 \frac{w}{\omega}=&\frac{ y^2}{2}-s(y^2+y)\cos(kx)+s^2\left[\frac{y^2+y}{4}+\left(\alpha_k( y^2+y)+\frac{\beta_k(y)}{2}\right)\cos(2kx)\right]+O(s^3)
 \end{align}
 in $C^{2+\alpha}(\ov{\0_0^b}).$
 The next goal is to find the expansion corresponding to $\psi_x.$
Recalling \eqref{eq:asy}, we have  
\begin{align*}
&\frac{\eta'}{(1+\eta)^2}=\eta'(1-2\eta)+O(s^3)=sk\sin(kx)+s^2k(1-2\alpha_k)\sin(2kx) +O(s^3)\\
&\frac{y-\eta}{1+\eta}=y(1-\eta)-\eta+O(s^2)=y+(y+1)s\cos(kx)+O(s^2),
\end{align*}
 and taking into account that $\psi(x,y)=w(x,(y-\eta)/(1+\eta))$ we find from 
\begin{align*}
\frac{\psi_x(x,y)}{\omega}=&\frac{w_x}{\omega}\left(x, \frac{y-\eta(x)}{1+\eta(x)}\right)-\frac{(1+y)\eta'}{(1+\eta)^2}\frac{w_y}{\omega}\left(x,\frac{y-\eta(x)}{1+\eta(x)}\right)
\end{align*}
that
\begin{align}\label{eq:AA}
\frac{\psi_x}{\omega}=&\frac{k}{2}\left(\frac{\sinh(2ky)}{\tanh(2k)}+\cosh(2ky)\right)\sin(2kx)s^2+O(s^3) \qquad \text{in} \ C^{2+\alpha}(\ov{\0_\eta^b}).
\end{align}
This is the key point in the proof of $(d)$ and $(e)$.

To keep the notation short we   introduce  the auxiliary function $f:[-1,0]\to\R$ defined by
 \[
f_k(y):=\frac{k}{2}\left(\frac{\sinh(2ky)}{\tanh(2k)}+\cosh(2ky)\right)=\frac{k\sinh(2k(1+y))}{2\sinh(2k)}, \qquad y\in[-1,0].
 \]
 Since $\psi_{xy}=s^2\omega f'_k(y)\sin(2kx)+O(s^3) $ and $f_k'$ is strictly positive in $[-1,0]$,
 we immediately obtain the assertion $(d)$.
 Additionally,   for arbitrary  $\delta\in (0,\pi/(8k))$, the 
derivative $\omega\psi_{xxy}(x,y)<0$ for all $(x,y)\in\ov\0_b$ with $x\in [\pi/(4k)+\delta,3\pi/(4k)-\delta],$  provided $\e_k$ is sufficiently small.
Recalling that $\psi=\omega/2$ on $y=-1,$ we get $\psi_{xx}(x,-1)=0$ and therefore $\omega\psi_{xx}(x,y)<0$ for all $x\in [\pi/(4k)+\delta,3\pi/(4k)-\delta] $ and $y>-1.$
On the other hand, the mixed derivative $\omega\psi_{xy}(\pi/(4k)+\delta,y)>0$ and $\omega\psi_{xy}(3\pi/(4k)-\delta,y)<0$ for all $y$, and taking into account that  $\psi_{x}(x,-1)=0,$ we conclude that there exists a point
 $\ov y_b\in(\min\eta,\max\eta)$ and a curve
 $x=\ov\xi(y):[-1,\ov y_b]\to (0,\pi/k)$ such that $(\ov \xi(\ov y_b),\ov y_b)$
is a point on the wave profile and 
 \begin{equation}\label{eq:psi3}
\psi_x(\ov\xi(y),y)=0, \qquad\text{for all  $y\in [-1,\ov y_b]$}.
 \end{equation}
In fact $y_b=y_\xi(\xi)$, and $\ov \xi(y)=\pi/(2k)+O(s)$ uniformly in $y\in[-1,\ov y_b].$
This curve $[x=\ov\xi(y)]$ splits the domain $\0_b$ into two subdomains  $\0_b\setminus[x=\ov\xi(y)]=\0_{b}^l\cup\0_{b}^r,$
$\0_{b}^l$ being the domain which has $x=0$ as  boundary component. 
Since $\Delta\psi_x=0$ in $\0_{b}^l$ and, since by $(c)$, $\omega\psi_x\geq0$ on  $\p\0_{b}^l$, the strong elliptic maximum principle ensures that $\omega\psi_x>0$ in $\0_{b}^l$.
The same argument shows that $\omega\psi_x<0$ in $\0_{b}^r$.

Finally, in order to show that $(\zeta,y_\zeta(\zeta)=(\ov\xi(\ov y_b),\ov y_b)$, we differentiate the relation $\psi(x,\eta(x))=0$ with respect to $x$ and obtain, by virtue of $\eta'>0$ in $(0,\pi/k),$
that, if $x\in(0,\pi/k), $ then   $\psi_x(x,y)=0$ if and only if $\psi_y(x,y)=0.$ 
We infer from $(c)$ that the desired equality holds and  the proof is completed.
 \end{proof}
 
We come now to the proof of Theorem \ref{MT:2}. 
It is based on Proposition \ref{P:2} and the fact that the curves obtained in Proposition \ref{P:2} $(c)$ and $(d)$ never intersect.

\begin{proof}[Proof of Theorem \ref{MT:2}]
 In order to precisely determine  the streamlines within the fluid located below we need to relate the two points $(\zeta,y_\zeta(\zeta)$ and $(\xi( y_b), y_b).$
Therefore, we differentiate the equation $\psi(x,\eta(x))=0$ twice with respect to $x$ and find, at $x=\zeta$, that
 \[
2\eta'(\zeta)\frac{\psi_{xy}(\zeta,y_\zeta(\zeta))}{\omega}=\left[-\frac{\psi_{xx}}{\omega}-\eta'^2\left(1-\frac{\psi_{xx}}{\omega}\right)\right](\zeta,y_\zeta(\zeta))
=-\frac{\psi_{xx}(\zeta,y_\zeta(\zeta))}{\omega}-\eta'^2(\zeta)+O(s^4).
\]
This relation is obtained by also using the relation  $\psi_y(\zeta,y_\zeta(\zeta))=0$ together with the equation $\Delta\psi=\omega$ in $\0_\eta^b,$  cf. \eqref{eq:psiDP}.
 Moreover, by \eqref{eq:asy2} and \eqref{eq:AA}, we find the following expansions
 \begin{align}\label{eq:AAA}
 &\eta'^2=k^2(1-\cos(2kx))s^2+O(s^3),\\
 &\frac{\psi_{xx}}{\omega}=k^2\left(\frac{\sinh(2ky)}{\tanh(2k)}+\cosh(2ky)\right)\cos(2kx)s^2+O(s^3),
\end{align}
 which yield, in the end,  
  \begin{align*}
2\eta'(\zeta)\frac{\psi_{xy}(\zeta,y_\zeta(\zeta))}{\omega}=-k^2\left\{1+\left[\left(\frac{\sinh(2ky)}{\tanh(2k)}+\cosh(2ky)\right)-1\right]\cos(2k\zeta)\right\}s^2+O(s^3).
\end{align*}
 Taking into account that  $(\zeta,y_\zeta(\zeta))\approx (\pi/(2k),0)$ and $\eta'(\zeta)>0$, we conclude that 
\[
\omega\psi_{xy}(\zeta,y_\zeta(\zeta))<0,
\]
and, by \eqref{eq:psi1}, 
\begin{equation}\label{eq:ara}
y_b<y_\zeta(\zeta).
\end{equation}
 This implies that the curve $[y=y_\zeta]$ is located entirely in the region $[\omega\psi_{xy}<0] $ and  that $y_\zeta$ is strictly increasing.
 Indeed, from the relation $\psi_y(x,y_\zeta(x))=0$ we find that
 \[
\psi_{xy}(x,y_\zeta(x))+y_{\zeta}'(x)\psi_{yy}(x,y_\zeta(x))=0
\]
 for all $x\in[\zeta,2\pi/k].$
 Since $\psi_{xy}(\zeta,y_\zeta(\zeta))<0$ and $\psi_{yy}>0,$ we conclude that $y_{\zeta} $ is  a strictly increasing function on $[\zeta,2\pi/k].$ 
 Whence, for $x\in(\zeta,\pi/k)$, the point $(x,y_\zeta(x))$ is located in the region  where $\psi_x<0,$ cf. \eqref{eq:psi2}.
 This means that the function $[\eta,\pi/k]\ni x\mapsto \omega\psi(x,y_\zeta(x)) $ attains its minimum in $x=\pi/k$ and, by $(c)$, this value is also the minimum of $\omega\psi:$
 \[
\min_{\0_b}\omega\psi=\omega\psi(\pi/k,y_\zeta(\pi/k)).
\]

We resume that there exist exactly three stagnation points one period: two on the wave surface $(\zeta,y_\zeta(\zeta))$ and  $(2\pi/k-\zeta,y_\zeta(2\pi/k-\zeta)),$ and one $(\pi/k,y_\zeta(\pi/k)$ beneath  the wave crest.
Since by Proposition \ref{P:2} $(c)$ and $(e)$ we know the sign of $\psi_x$ and $\psi_y$ in the whole domain $\0_\eta^b$, we conclude that the streamlines of the flow are as in Figure \ref{F:1}.
\end{proof}


\begin{thebibliography}{99}  

\bibitem{A} {{H. Amann}: {\em  Linear and Quasilinear Parabolic Problems, Volume I Abstract Linear Theory}, Birkh\"auser, Basel, 1995.}

  \bibitem{BLS}
 {J. L. Bona, D. Lannes \& J.-C. Saut }:
Asymptotic models for internal waves,
\textit{J. Math. Pures Appl.} 89 (2008), 538--566.


 
\bibitem{C2}
 {A. Constantin}:
The trajectories of particles in Stokes waves,
\textit{Invent. Math.} {\bf 166} (2006), 523--535.


 \bibitem{CE3} 
 {A. Constantin and J. Escher}:
Analyticity of periodic traveling free surface water waves with vorticity,
\textit{Ann. of Math.} {\bf 173} (2011), 559--568.


\bibitem{CJ}
{A. Constantin and R. S. Johnson}:
Propagation of very long water waves, with vorticity, over variable
depth, with applications to tsunamis, 
{\em Fluid Dynam. Res.} {\bf 40} (2008),
175--211.


\bibitem{CS2} 
{A. Constantin and W. Strauss}:
Exact steady periodic water waves with vorticity,
\textit{Comm. Pure Appl. Math.} {\bf 57}(4) (2004), 481--527. 


\bibitem{CS3}
{A. Constantin and W.  Strauss}: Pressure beneath a Stokes wave,
  \textit{Comm. Pure Appl. Math. } {\bf  63}(4) (2010),  533--557.





\bibitem{CS41} 
{A. Constantin and W. Strauss}: {\em Periodic traveling gravity water waves with discontinuous vorticity},
{\it Arch. Ration. Mech. Anal.} {\bf 202} (1) (2011), 133--175.



\bibitem{CV2}
{A. Constantin and E. Varvaruca}:
Steady periodic water waves with constant vorticity: Regularity and local bifurcation, {\em Arch. Rational Mech. Anal. } {\bf 199} (2011), 33--67.
  

\bibitem{CV}
 {A. Constantin and G. Villari}: Particle trajectories in linear water waves,
 \textit{J. Math. Fluid Mech.}   {\bf 10} (1)  (2008), 1--18.

\bibitem{CR}   
{M.~G. Crandall and P.~H. Rabinowitz}:  Bifurcation from simple eigenvalues,
    {\em J. Funct. Anal.} {\bf{8}} (1971), 321--340. 
    

\bibitem{DG} {{G. DaPrato and P. Grisvard}: Equations d'\'evolution abstraites nonlin\'eaires de type
parabolique, {\em  Ann. Mat. Pura Appl.} {\bf{4}} (120)  (1979), 329--396.}

 \bibitem{EEV}
{M. Ehrnstr\"om, J. Escher, and G. Villari}: Steady water waves with multiple critical layers: Interior dynamics,
to appear in {\em J. Math. Fluid Mech.}


\bibitem{EEW}
{M. Ehrnstr\"om, J. Escher, and E. Wahl\'en}: Steady water waves with multiple critical layers,
{\em SIAM J. Math. Anal.}  {\bf 43} (2011), 1436--1456. 

\bibitem{Mats} 
{M. Ehrnstr\"om and G. Villari}:  Linear water waves with vorticity: Rotational features and particle paths,
 {\em J. Differential Equations} {\bf 244} (2008), 1888--1909.
 
\bibitem{EW}
{M. Ehrnstr\"om  and E. Wahl\'en}: On steady water waves with critical layers, preprint. 
  
\bibitem{ES}    { {J. Escher and G. Simonett}: Classical solutions for Hele-Shaw models with surface tension,
{\em Adv. Differential Equations }{\bf 2} (1997) 619--642.}

\bibitem{EM} 
{J. Escher and B.-V.  Matioc}:  A moving boundary problem for periodic Stokesian Hele-Shaw flows,
 {\em Interfaces Free Bound.} {\bf 11} (2009), 119--137.

\bibitem{EMM2} 
{J. Escher,  A.--V. Matioc and B.--V. Matioc }:  On stratified steady periodic water waves with linear density distribution  and stagnation points,
  {\em J. Differential Equations} {\bf 251} (10) (2011), 2932--2949.


\bibitem{HM}
 {K.R. Helfrich and  W.K. Melville}:
Long nonlinear internal waves,
\textit{Annual Review of Fluid Mechanics}  {\bf 38} (2006), 395--425.

 \bibitem{DH2}
{D. Henry}:
Analyticity of streamlines for periodic
traveling free surface capillary-gravity water
waves with vorticity, 
 \textit{SIAM J. Math. Anal.} {\bf 42} (6), (2010), 3103--3111.
 
 \bibitem{DH4}
{D. Henry}:
Analyticity of the free surface for periodic
traveling  capillary-gravity water
waves with vorticity,  J. Math. Fluid Mech. (2011), DOI 10.1007/s00021-011-0056-z.

\bibitem{Hen2} {D. Henry}: 
Regularity for steady periodic capillary water waves with vorticity,  
to appear in \textit{Philos. Trans. R. Soc. Lond. Ser. A}.


\bibitem{DB} 
{D. Henry and B.-V.  Matioc}:  On the existence of steady periodic capillary-gravity  stratified water waves,
to appear in {\it Ann. Sc. Norm. Super. Pisa Cl. Sci.}.

\bibitem{IK}
 {D. Ionescu-Kruse}: Elliptic and hyperelliptic functions
describing the particle motion beneath
small-amplitude water waves with constant
vorticity, preprint.


\bibitem{KNS}
 {D. Kinderlehrer, L. Nirenberg,  and J. Spruck}:
 Regularity in elliptic free boundary value problems I,
  \textit{J. Anal. Math.} {\bf 34} (1978), 86--119. 
  
  
  \bibitem{Ligh} 
{J. Lighthill}: {\em Waves in fluids}, Cambridge University Press, Cambridge, 1978.


\bibitem{L} 
{A. Lunardi}:  {\em Analytic Semigroups and Optimal Regularity in Parabolic Problems},  Birkh\"auser, Basel, 1995.


\bibitem{AM}
 {A.-V. Matioc}: 
On particle trajectories in linear water waves,
  \textit{ Nonlinear Anal. Real World Appl.} {\bf 11} (5) (2010), 4275--4284. 


 \bibitem{M1}
{B.-V. Matioc}:
Analyticity of the streamlines for periodic traveling water waves with bounded vorticity.
{\em Int. Math. Res. Not.} {\bf 17} (2011), 3858--3871.


\bibitem{M} 
{J.~B. McLeod}:
 The Stokes and Krasovskii conjectures for the wave of greatest height, 
{\em Stud. Appl. Math.} {\bf 98} (1997), 311--334. 

\bibitem{P}
{P. I. Plotnikov}: Proof of the Stokes conjecture in the theory of
surface waves, {\em Stud. Appl. Math.} {\bf 108} (2) (2002), 217--244.


 \bibitem{T}
  {J.~F. Toland}: Stokes waves, {\em Topol. Methods Nonlinear Anal.} {\bf 7 [8]} (1996 [1997]), 1--48 [412--414].
  

\bibitem{W}
{E. Wahl\'en}: Steady water waves with a critical layer, {\em J. Differential Equations} {\bf 246} (2009), 2468--2483.
\end{thebibliography}
  \end{document}